\documentclass[12pt]{article}
\usepackage[letterpaper,text={6.5in,9in}]{geometry}

\usepackage{amssymb,latexsym,amsmath,epsfig,amsthm} 
\usepackage{color}
\usepackage[colorlinks=true,citecolor=black,linkcolor=black,urlcolor=blue]{hyperref}
\usepackage{graphicx}
\usepackage{array}
\usepackage{soul}

\newcommand{\integers}{{\mathbb Z}}

\newcommand{\B}[2]{B^{(#1)}_{#2}}
\newcommand{\F}{{\mathbb F}}
\newcommand{\cM}{{\cal M}}
\newcommand{\Wlog}{W.l.o.g.}

\newcommand\undermat[2]{%
  \makebox[0pt][l]{$\smash{\underbrace{\phantom{%
    \begin{matrix}#2\end{matrix}}}_{\text{$#1$}}}$}#2}

\DeclareMathOperator{\supp}{supp}

\newtheorem{lemma}{Lemma}

\newtheorem{proposition}{Proposition}
\newtheorem{corollary}{Corollary}
\newtheorem{claim}{Claim}

\newcommand{\bF}{\mathbb{F}}
\newcommand{\row}{{\bf row}}

\begin{document}

\title{A Combinatorial Problem Solved by a Meta-Fibonacci Recurrence Relation}

\author{Ramin Naimi \qquad Eric Sundberg \\
\small Mathematics Department \\[-0.8ex]
\small Occidental College \\[-0.8ex]
\small Los Angeles, CA \\[-0.8ex]
\small\tt \{rnaimi, sundberg\}@oxy.edu 
}


\maketitle


\begin{abstract}
We present a natural, combinatorial problem whose solution is given by the 
meta-Fibonacci recurrence relation $a(n) = \sum_{i=1}^p  a(n-i+1 - a(n-i))$,
where $p$ is prime.  This combinatorial problem is less general than those
given in \cite{jacrus} and \cite{rusdeu}, but it has the advantage of having a simpler
statement.

\end{abstract}

\section{Introduction}

Let $M$ be a matrix with entries in $\integers_2$, such that every column contains
at least one~1.  We want to pick a subset of the rows such that when they are added
together modulo~2, their sum $\vec{s}$ has as many~1's as possible.  If $M$
has $n$ columns, what is the largest number of 1's we can guarantee 
$\vec{s}$ to have?  For example, if $n=5$, we can always find a set of  rows
whose sum $\vec{s}$ contains at least four 1's.
Let $\lambda(n)$ denote the largest number of 1's $\vec{s}$ can be guaranteed to have for
any $M$ with $n$ nonzero
columns.  We will show that $\lambda(n)$ satisfies the recurrence
relation 
\begin{equation}\label{mf_eqn}
\lambda(n) = \lambda(n-\lambda(n-1)) + \lambda(n-1 -\lambda(n-2)).
\end{equation}
More generally, for $p$ prime,  let $\vec{v} = (v_1, \ldots, v_n)$ satisfy $v_i \in\F_p$ for $1\leq i \leq n$.
Let $\supp(\vec{v}) = \{i \in [n]: v_i \neq 0\}$ and let $\|\vec{v}\| = |\supp(\vec{v})|$, 
i.e., $\|\vec{v}\|$ is the number of nonzero terms in $\vec{v}$.
Let $M$ be an $m\times n$ matrix whose entries are in $\F_p$.  Let $\row(M)$ be the 
rowspace of $M$, i.e., the set of all linear combinations of the row vectors of $M$
over the field $\F_p$.
Let $c(M)$ denote the {\em capacity} of $M$, which we define as follows,
$$c(M) = \max_{\vec{v} \in \row(M)} \|\vec{v}\|.$$
For each integer $n \geq 1$, let $\lambda_p(n)$ 
be the minimum possible capacity of an $\bF_p$-matrix
consisting of $n$ {\em nonzero} columns (i.e., no column equals $\vec{0}$).  
Restated, let
$$\cM_n^* = \{M \in \F_p^{m\times n} : 1 \leq m \leq p^n \mbox{ and no column of $M$ equals $\vec{0}$}\},$$
then 
$$\lambda_p(n) = \min_{M\in\cM_n^*} c(M).$$  
We will see that $\lambda_p$ satisfies the recurrence relation
\begin{equation}\label{mfp_eqn}
\lambda_p(n) = \sum_{i=1}^p  \lambda_p(n-i+1 -\lambda_p(n-i)).
\end{equation}
This type of recurrence relation is called a meta-Fibonacci relation.

Meta-Fibonacci sequences have been studied by various authors, dating at least as far back
as 1985, when Hofstadter~\cite{hof} apparently coined the term ``meta-Fibonacci."  
These are integer sequences defined by ``nested, Fibonacci-like"
recurrence relations, such as relation~\eqref{mf_eqn}, 
which was studied by Conolly~\cite{con},
and  \eqref{mfp_eqn}.
Generalizations of \eqref{mfp_eqn} were shown in \cite{jacrus} and \cite{rusdeu} to be
solutions to certain combinatorial problems involving $k$-ary infinite trees, 
and compositions of integers.   
The ``matrix capacity'' problem described above is a different combinatorial problem whose solution is also given by  relation~\eqref{mfp_eqn}.
This combinatorial problem is ``natural"  in the sense that 
it arose while the first named author was working on a problem in spatial graph theory.   
It was only later that we learned (through the OEIS \href{http://oeis.org/A046699}{A046699}) 
that it can be characterized as a meta-Fibonacci sequence.

\section{Main Result}

We begin with a lemma which allows us to produce a lower bound on $\lambda_p(n)$.  
For the remainder of this paper, instead of writing $\lambda_p$, we will simply write $\lambda$. 
For a matrix $M$,
let $\row^*(M) = \row(M) -\{\vec{0}\}.$
\begin{lemma}\label{incr_lemma}
Let $M$ be an $\bF_p$-matrix with $n$ nonzero columns, i.e., $M\in\cM_n^*$.  
Let $\vec{v} \in \row^*(M).$  If $$p \lambda(n-\|\vec{v}\|) > \|\vec{v}\|,$$
then there is a vector $\vec{z}\in \row^*(M)$ such that $\|\vec{z}\| > \|\vec{v}\|.$
\end{lemma}

\begin{proof}
Let $M$ be an $\bF_p$-matrix with $n$ nonzero columns.  
Let $\vec{v} \in \row^*(M)$, and let $k = \|\vec{v}\|.$ Let $\vec{v} = (v_1, \ldots, v_n).$  \Wlog, 
suppose $v_i \neq 0$ for $1 \leq i \leq k$ 
and $v_i = 0$ for $k+1\leq i \leq n$. 
Let $\vec{w}\in\row^*(M)$ be such that $w_i \neq 0$ for at least $\lambda(n-k)$  coordinates $i$, where $k+1 \leq i \leq n$.
In other words, if we let $\vec{w}_L  = (w_1, \ldots, w_k)$ and $\vec{w}_R = (w_{k+1}, \ldots, w_n)$, then $\|\vec{w}_R\| \geq \lambda(n-k)$.
Since $\|\vec{w}\| = \|\vec{w}_L\|+\|\vec{w}_R\|$,
if $\|\vec{w}_L\| \geq (p-1)\lambda(n-k)$, then $\|\vec{w}\| \geq p \lambda(n-k) > \|\vec{v}\|$, and we are done.
So we may assume that $\|\vec{w}_L\| < (p-1)\lambda(n-k).$

Our goal will be to prove that there exists a nonzero constant $c$ such that 
$\|c \vec{w}_L + \vec{v}_L\|  > k - \lambda(n-k)$, where $\vec{v}_L = (v_1, \ldots, v_k).$  Once we establish
that such a constant exists, then we will be done, because we will have 
$\|c\vec{w} + \vec{v}\| = \|c \vec{w}_L + \vec{v}_L\| + \|\vec{w}_R\| > (k-\lambda(n-k)) + \lambda(n-k) = k.$


For $1 \leq a \leq p-1$,  let $S_a = \{i \in [k] : a w_i + v_i = 0\}$. 
Since $v_i \neq 0$ for $1 \leq i \leq k$, then $S_a \subseteq \supp(\vec{w}_L).$ 
Thus, if $a w_i + v_i = 0 = b w_i + v_i$,
then $w_i \neq 0$, which allows us to conclude that $a=b$. Therefore, if $a \neq b$, then 
$S_a \cap S_b = \emptyset$.  Since $\bigcup_{a=1}^{p-1} S_a \subseteq \supp(\vec{w}_L)$
and the $S_a$ are pairwise disjoint, we have 
$$\sum_{a=1}^{p-1} |S_a| \leq |\supp(\vec{w}_L)| = \|\vec{w}_L\|<(p-1)\lambda(n-k).$$
Therefore, the average value of $|S_a|$ is strictly less than $\lambda(n-k)$, and if we let $c\in[p-1]$ be such that
$|S_c|$ is minimum, then $|S_c| < \lambda(n-k)$.
Thus, $\|c \vec{w}_L + \vec{v}_L\| = k - |S_c| > k - \lambda(n-k)$, and
as noted above, we are done.
Specifically, $\|c\vec{w}+\vec{v}\| > \|\vec{v}\|.$
\end{proof}


It is easy to check that the following corollary holds. 
\begin{corollary}\label{small_n}
If $1 \leq n \leq p$, then $ \lambda(n) = n.$
\end{corollary}

For an integer $k\geq 0$, let $\sigma_k = \sum_{j=0}^k p^j.$

\begin{proposition}\label{lin_comb}
Suppose $$n = \sum_{j=\ell}^k b_j \sigma_j,$$
where $b_k \geq 1$, and $0 \leq b_j \leq p-1$ for $j \neq \ell$, and $1 \leq b_\ell \leq p$.
Then 
$$\lambda(n) \geq  \sum_{j=\ell}^k b_j p^j.$$
\end{proposition}
\begin{proof}[Proof of Proposition~\ref{lin_comb}]
We proceed by induction on $k$.  When $k=0$, then $n = b_0 \sigma_0 = b_0.$  Since $1 \leq b_0 \leq p$,
then $\lambda(n) = b_0$ 
by Corollary~\ref{small_n}, thus, $\lambda(n) = b_0 p^0$
and the result holds.  Now suppose $k \geq 1$.  Our inductive hypothesis
will be if
$$n = \sum_{j=\ell}^m b_j \sigma_j,$$
where $b_m \geq 1$, and $0 \leq b_j \leq p-1$ for $j \neq \ell$, and $1 \leq b_\ell \leq p$, and $m < k$, then
$$\lambda(n) \geq  \sum_{j=\ell}^m b_j  p^j.$$  Let $M$  be an $\bF_p$-matrix with $n$ nonzero columns.
Suppose $\vec{v} \in \row^*(M)$ with $$\|\vec{v}\| < \sum_{j=\ell}^k b_j p^j.$$  Then
\begin{align*}
n-\|\vec{v}\| > n - \sum_{j=\ell}^k b_j p^j  &= \sum_{j=\ell}^k b_j \sigma_j - \sum_{j=\ell}^k b_j p^j \\
							     &= \sum_{j=\ell}^k b_j (\sigma_j - p^j) \\
							     &= \sum_{j=\ell}^k b_j \sigma_{j-1},
\end{align*}
where we define $\sigma_{-1}= 0$ to handle the case $j=0$, since $\sigma_0 - p^0 = 0.$  Thus,
$$n-\|\vec{v}\| \geq \sum_{\ell -1 \leq j \leq k-1} b_{j+1} \sigma_j +1.$$  We want to determine a lower
bound on $p\lambda\left(\sum_{j=\ell -1}^{k-1} b_{j+1} \sigma_j +1\right)$ that allows us
to conclude that $p\lambda(n-\|\vec{v}\|) > \|\vec{v}\|$ so that we may use Lemma~\ref{incr_lemma}.
We consider the case where $b_\ell = p$  and the case where $1 \leq b_\ell \leq p-1$
separately.

Suppose $b_\ell = p$.  Then
\begin{align*}
\sum_{\ell -1 \leq j \leq k-1} b_{j+1} \sigma_j +1 
     &= \sum_{\ell \leq j \leq k-1} b_{j+1} \sigma_j + b_\ell \sigma_{\ell-1}+1 \\
     &= \sum_{\ell \leq j \leq k-1} b_{j+1} \sigma_j + (p \sigma_{\ell-1}+1) \\
     &= \sum_{\ell \leq j \leq k-1} b_{j+1} \sigma_j + \sigma_{\ell} \\
     &= \sum_{\ell+1 \leq j \leq k-1} b_{j+1} \sigma_j + b_{\ell+1} \sigma_{\ell}+\sigma_\ell \\
     &= \sum_{\ell+1 \leq j \leq k-1} b_{j+1} \sigma_j + (b_{\ell+1}+1) \sigma_{\ell}.
\end{align*}
Notice that our sum satisfies all of the criteria for the inductive hypothesis. Specifically,
the coefficient of its lowest sigma-term $\sigma_\ell$ is $b_{\ell+1} +1$, which satisfies
$1 \leq b_{\ell+1} +1 \leq p$; the coefficient of $\sigma_j$ is $b_{j+1}$ and $0 \leq b_{j+1} \leq p-1$
for $j \neq \ell$; the coefficient of the largest sigma-term $\sigma_{k-1}$ is $b_k$, which 
satisfies $b_k \geq 1$; and finally, the index of its largest sigma term is $k-1$ which is strictly
less than $k$.  Therefore, by the inductive hypothesis,
\begin{align*}
p \cdot\lambda\left(\sum_{\ell+1 \leq j \leq k-1} b_{j+1} \sigma_j + (b_{\ell+1}+1) \sigma_{\ell}\right)
  & \geq  p \left(\sum_{\ell+1 \leq j \leq k-1} b_{j+1} p^j + (b_{\ell+1}+1) p^\ell\right) \\
  &= \sum_{\ell+1 \leq j \leq k-1} b_{j+1} p^{j+1} + (b_{\ell+1}+1) p^{\ell+1}  \\
  &= \sum_{\ell \leq j \leq k-1} b_{j+1} p^{j+1} + p \cdot p^\ell  \\
  &= \sum_{\ell+1 \leq j \leq k} b_{j} p^{j} + p \cdot p^\ell  \\
  &=\sum_{\ell \leq j \leq k} b_{j} p^{j},
\end{align*}
where the last equality holds because $b_\ell = p$.  Since $\lambda$  is a nondecreasing function,
our previous work implies
\begin{align*}
p\lambda(n-\|\vec{v}\|) 
    &\geq p \cdot\lambda\left(\sum_{\ell+1 \leq j \leq k-1} b_{j+1} \sigma_j + (b_{\ell+1}+1) \sigma_{\ell}\right)\\
    & \geq \sum_{\ell \leq j \leq k} b_{j} p^{j} > \|\vec{v}\|.
\end{align*}
Thus, by Lemma~\ref{incr_lemma}, there is a vector $\vec{z}\in \row^*(M)$ such that $\|\vec{z}\| > \|\vec{v}\|.$

Now suppose $1 \leq b_\ell \leq p-1$.  Recall that our sum is 
$$\sum_{\ell -1 \leq j \leq k-1} b_{j+1} \sigma_j +1\ 
 = \sum_{\ell -1 \leq j \leq k-1} b_{j+1} \sigma_j +1\cdot\sigma_0.$$  
In this case, the smallest sigma-term
is $\sigma_0$, and its coefficient is $b_1+1$, where $b_1= 0$ if $\ell \geq 2$.  We note that
our sum satisfies all of the criteria for the inductive hypothesis.
Since each $b_j$ satisfies $0 \leq b_j \leq p-1$, then $1 \leq b_1 + 1 \leq p$;
when $j \geq 1$, the coefficient of each $\sigma_j$ is $b_{j+1}$ and $0 \leq b_{j+1} \leq p-1$;
the coefficient of the largest sigma-term $\sigma_{k-1}$ is $b_k$, which 
satisfies $b_k \geq 1$; and finally, the index of its largest sigma term is $k-1$ which is strictly
less than $k$.  

When $\ell \geq 2$, 
the coefficient of $\sigma_0$ is 1, and we apply
 the inductive hypothesis to obtain
 \begin{align*}
 p \cdot\lambda\left(\sum_{\ell-1 \leq j \leq k-1} b_{j+1} \sigma_j +1\right)
    & \geq p \left(\sum_{\ell-1 \leq j \leq k-1} b_{j+1} p^j +1\right) \\
    & = \sum_{\ell-1 \leq j \leq k-1} b_{j+1} p^{j+1} +p \\
    & = \sum_{\ell \leq j \leq k} b_{j} p^{j} +p.
 \end{align*}
Thus, 
\begin{align*}
p\lambda(n-\|\vec{v}\|) 
    &\geq  p \cdot\lambda\left(\sum_{\ell-1 \leq j \leq k-1} b_{j+1} \sigma_j +1\right)\\
    & \geq \sum_{\ell \leq j \leq k} b_{j} p^{j} + p > \|\vec{v}\|.
\end{align*}
When $\ell \in\{0,1\}$, 
our sum is $\sum_{j=0}^{k-1} b_{j+1} \sigma_j +1$, and we apply the inductive hypothesis
to obtain
\begin{align*}
 p \cdot\lambda\left(\sum_{0 \leq j \leq k-1} b_{j+1} \sigma_j +1\right)
    &=p \cdot\lambda\left(\sum_{1 \leq j \leq k-1} b_{j+1} \sigma_j +(b_1+1)\right) \\
    & \geq p \left(\sum_{1 \leq j \leq k-1} b_{j+1} p^{j} +b_1+1\right) \\
    & = \sum_{1 \leq j \leq k-1} b_{j+1} p^{j+1} +b_1p+p \\
    & = \sum_{1\leq j \leq k} b_{j} p^{j} +p.
 \end{align*}
 Thus, 
\begin{align*}
p\lambda(n-\|\vec{v}\|) 
    &\geq  p \cdot\lambda\left(\sum_{0 \leq j \leq k-1} b_{j+1} \sigma_j +1\right)\\
    & \geq \sum_{1 \leq j \leq k} b_{j} p^{j} + p \\
    & > \sum_{\ell \leq j \leq k} b_{j} p^{j}  > \|\vec{v}\|.
\end{align*}
Thus, by Lemma~\ref{incr_lemma}, there is a vector $\vec{z}\in \row^*(M)$ such that $\|\vec{z}\| > \|\vec{v}\|.$
Therefore $\lambda(n) \geq  \sum_{j=\ell}^k b_j p^j.$
\end{proof}


Now we show that every $n \geq 1$ can be written in the form described in Proposition~\ref{lin_comb}.

\begin{claim}\label{n_as_lin_comb_of_sigmas}
Let $n \in\integers^+$.  Suppose $n < \sigma_{k+1}$. 
Let $n_{k+1} = n$, and for $0\leq j \leq k$, assuming $n_{j+1}$ is defined,
let $b_j$ be the largest integer such that $b_j\sigma_j \leq n_{j+1}$, and let $n_j = n_{j+1} - b_j\sigma_j$.
Then for $0 \leq j \leq k$, we have $0 \leq n_{j+1} \leq p\sigma_j$ and $0 \leq b_j \leq p$.
Moreover, 
 $$n = \sum_{j=0}^k b_j \sigma_j,$$
and if $b_j = p$, then $b_i = 0$ for $i < j$.
\end{claim}
\begin{proof}[Proof of Claim~\ref{n_as_lin_comb_of_sigmas}]
Suppose $n \in\integers^+$ and $n < \sigma_{k+1}$.  Then $n \leq \sigma_{k+1} -1= p\sigma_k$.
Let $n_{k+1} = n$, and for $0\leq j \leq k$, assuming $n_{j+1}$ is defined,
let $b_j$ be the largest integer such that $b_j\sigma_j \leq n_{j+1}$, and let $n_j = n_{j+1} - b_j\sigma_j$.
We proceed by  induction on  $k-j$.  Assume $0\leq n_{j+1} \leq p\sigma_j$ and let $b_j$ and $n_j$
be defined as above.  Since $0 \leq n_{j+1}$, then $b_j \geq 0$.  Since $n_{j+1} \leq p\sigma_j$ and 
$b_j\sigma_j \leq n_{j+1}$, then $b_j\sigma_j \leq p\sigma_j$. Thus, since $\sigma_j \geq 1$, we have
$b_j \leq p.$  Since $b_j\sigma_j \leq n_{j+1}$ and $n_j = n_{j+1}-b_k\sigma_k$, then $n_j \geq 0.$
Since $n_{j+1} < (b_j+1)\sigma_j$, then $n_{j+1}-b_j\sigma_j < \sigma_j$, i.e., 
$n_j \leq \sigma_j-1 = p\sigma_{j-1}$.
Therefore, by induction, $0 \leq n_{j+1} \leq p\sigma_j$ and $0 \leq b_j \leq p$  for $0 \leq j \leq k$.

Now suppose $b_j = p$.  Since  $b_j \sigma_j \leq n_{j+1} \leq p \sigma_j$, then $n_{j+1} = p\sigma_j$ and $n_j = n_{j+1} - b_j\sigma_j =0$.
Moreover, $b_i = 0$ and $n_i = 0$ for all $i<j$.

To see that $n = \sum_{j=0}^k b_j \sigma_j,$ observe that $b_j\sigma_j = n_{j+1}-n_j$ for $0\leq j\leq k$, because of the definition of $n_j$.
Thus, 
$$\sum_{j=0}^k b_j \sigma_j = \sum_{j=0}^k (n_{j+1}-n_j) = n_{k+1}-n_0 = n-n_0.$$
Since $0 \leq n_1 \leq p\sigma_0 = p$, then, by definition, $b_0 = n_1$ and $n_0 = n_1 - b_0\sigma_0 = n_1-n_1(1) = 0.$
Thus, $\sum_{j=0}^k b_j \sigma_j= n$
\end{proof}

With Proposition~\ref{lin_comb} and Claim~\ref{n_as_lin_comb_of_sigmas}, we have established a lower bound
on $\lambda(n)$ for all $n \geq 1$.  We need to prove the corresponding upper bound.  We will do so by 
constructing a matrix with $n$ columns whose capacity equals the lower bound given in Proposition~\ref{lin_comb}.
We begin by constructing such a matrix for certain values of $n$, namely, when $n = \sigma_k$ for some $k \geq 0$.

For each integer $k\geq 0$, we define a $(k+1)\times\sigma_k$ matrix $B_k$, recursively, as follows.  
The matrix $B_0$ is the $1\times 1$ matrix whose sole entry is 1.  For $k\geq 1$, $B_k$ can be defined
as a block matrix with a ``row" consisting of $p$ copies of $B_{k-1}$ followed by a $k\times 1$ column of 0's, 
then one more row of dimensions $1\times \sigma_k$ with its first $\sigma_{k-1}$ entries equal to 0 (below the first
$B_{k-1}$), then $\sigma_{k-1}$ entries equal to 1 (below the next $B_{k-1}$), \ldots, then $\sigma_{k-1}$
entries equal to $p-1$ (below the last $B_{k-1}$), and one last entry equal to 1, i.e.,
$$B_k = \left[\begin{array}{c|c|c|c|c}
\text{\Large $B_{k-1}$} & \text{\Large $B_{k-1}$} & \cdots 
& \text{\Large $B_{k-1}$} & \begin{array}{c}0 \\ \vdots \\0 \end{array} \\
\hline
0\ldots 0 & 1\ldots 1 & \cdots & (p-1) \ldots (p-1) & 1
\end{array}\right].$$
For $k\geq 1$, let $B'_k$ be the $k\times \sigma_{k}$ matrix obtained from $B_k$ by removing its last row,
i.e., 
$$B'_k = \left[\begin{array}{c|c|c|c|c}
\text{\Large $B_{k-1}$} & \text{\Large $B_{k-1}$} & \cdots 
& \text{\Large $B_{k-1}$} & \begin{array}{c} 0 \\ \vdots \\ 0 \end{array} 
\end{array}\right].$$

\begin{lemma}\label{Bk_lemma}
For each $\vec{v} \in\row^*(B_k)$, $\|\vec{v}\| = p^k$.  
\end{lemma}

\begin{proof}
We proceed by induction on $k$.  When $k=0$, the result is trivial. 
Let $k\geq 1$.  Assume the result for $j<k$.  Let $\vec{v} \in \row^*(B_k)$.
We first consider the case where $\vec{v} \in \row^*(B'_k)$.
Then we can write 
$$\vec{v} = 
(v^{(0)}_{1}, \ldots, v^{(0)}_{\sigma_{k-1}}, v^{(1)}_{1}, \ldots, v^{(1)}_{\sigma_{k-1}}, \ldots, v^{(p-1)}_{1}, \ldots, v^{(p-1)}_{\sigma_{k-1}}, 0).$$
To shorten notation, we will  write
\begin{equation}\label{short_notation}
\vec{v} = (\vec{v}_0, \vec{v}_1, \ldots, \vec{v}_{p-1}, 0),
\end{equation}
where $\vec{v}_i = (v^{(i)}_{1}, \ldots, v^{(i)}_{\sigma_{k-1}})$ for $0 \leq i \leq p-1$. Technically, 
in equation~\eqref{short_notation}, $\vec{v}_i$ simply represents the coordinates $v^{(i)}_{1}, \ldots, v^{(i)}_{\sigma_{k-1}}$.
We observe that  $\vec{v}_0 = \vec{v}_1 = \cdots = \vec{v}_{p-1}$
based on how $B'_k$ and $\vec{v}$ are defined. 
We also observe that $\vec{v}_i \in \row^*(B_{k-1})$. By the inductive hypothesis, 
$\|\vec{v}_i\| = p^{k-1}$, therefore, $\|\vec{v}\| = p^k$.

We now show the result holds for  $\vec{w} \in \row^*(B_k)-\row^*(B'_k)$.  Let $\vec{u}$ be the last row in $B_k$,
i.e., $\vec{u} = (0, \ldots, 0, 1, \ldots, 1, \ldots, p-1, \ldots, p-1,1)$.  
We observe that $\|\vec{u}\| = \sigma_k-\sigma_{k-1} = p^k$,
thus, the result holds when $\vec{w}  = \vec{u}.$  To illustrate our argument, we next consider the special case where
$\vec{w} = \vec{v} + \vec{u}$ for some $\vec{v} \in \row^*(B'_k)$.  
Again, we slightly abuse notation and write $\vec{u} = (\vec{0}, \vec{1}, \ldots, (p-1)\vec{1}, 1)$, where $\vec{c}$ (or $c\vec{1}$)
represents the $\sigma_{k-1}$-dimensional vector $(c, \ldots, c)$.
Then we can write $\vec{v} + \vec{u} = (\vec{v}_0+\vec{0}, \vec{v}_1+\vec{1}, \ldots, \vec{v}_{p-1}+(p-1)\vec{1}, 1)$.
Since we are working modulo~$p$, a coordinate of $\vec{v}_j+ j\vec{1}$ is congruent to 0 if and only if the corresponding coordinate
of $\vec{v}_j$ is congruent to $p-j$.
Thus, we can  count 
the total number of coordinates
that are congruent to 0  in $\vec{v} + \vec{u}$  
as follows
\begin{equation}\label{count_zeros}
\left(\begin{array}{c}
\mbox{Total \# of 0-coordinates} \\
\mbox{in $\vec{v}+ \vec{u}$}
\end{array}  \right) = \sum_{j=0}^{p-1} ( \textrm{\# of $(p-j)$-coordinates in $\vec{v}_j$}).
\end{equation}
Since $\vec{v}_0 = \vec{v}_1 = \cdots = \vec{v}_{p-1}$, equation~\eqref{count_zeros}
reduces to 
\begin{equation*}\label{count_zeros_simple}
\left(\begin{array}{c}
\mbox{Total \# of 0-coordinates} \\
\mbox{in $\vec{v}+\vec{u}$}
\end{array}  \right) =  
\left(\begin{array}{c}
\mbox{Total \# of coordinates} \\
\mbox{in $\vec{v}_0$}
\end{array}  \right) = \sigma_{k-1}.
\end{equation*}
Thus, $\|\vec{v}+\vec{u}\| = \sigma_k - \sigma_{k-1}= p^k.$  In general, $\vec{w} \in \row^*(B_k)-\row^*(B'_k)$
satisfies $\vec{w} = \vec{v} + c\vec{u}$ for some $\vec{v} \in \row^*(B'_k)$ and $c \not\equiv 0 \pmod{p}$.
In this case, $\vec{w} = (\vec{v}_0+c\vec{0}, \vec{v}_1+c\vec{1}, \ldots, \vec{v}_{p-1}+c(p-1)\vec{1}, 1)$, and
equation~\eqref{count_zeros} becomes
\begin{equation}\label{new_count_zeros}
\left(\begin{array}{c}
\mbox{Total \# of 0-coordinates} \\
\mbox{in $\vec{w}$}
\end{array}  \right) = \sum_{j=0}^{p-1} ( \textrm{\# of $(p-cj)$-coordinates in $\vec{v}_j$}),
\end{equation}
where arithmetic is modulo~$p$.  Since $\vec{v}_0 = \vec{v}_1 = \cdots = \vec{v}_{p-1}$,
we obtain
$$
\left(\begin{array}{c}
\mbox{Total \# of 0-coordinates} \\
\mbox{in $\vec{w}$}
\end{array}  \right) = \sum_{j=0}^{p-1} ( \textrm{\# of $(p-c j)$-coordinates in $\vec{v}_0$}).
$$
Since $p$ is prime and $c\not\equiv 0 \pmod{p}$, then
$\{p, p-c, p-2c, \ldots, p-(p-1)c\}$ is a equivalent to
$\{0, 1, \ldots, p-1\}$ modulo~$p$, thus,
$$
\left(\begin{array}{c}
\mbox{Total \# of 0-coordinates} \\
\mbox{in $\vec{w}$}
\end{array}  \right) =  
\left(\begin{array}{c}
\mbox{Total \# of coordinates} \\
\mbox{in $\vec{v}_0$}
\end{array}  \right) = \sigma_{k-1}.
$$
Therefore, $\|\vec{w}\| = \sigma_k-\sigma_{k-1} = p^k$,
and we can conclude that for each $\vec{v} \in\row^*(B_k)$, $\|\vec{v}\| = p^k$.
\end{proof}

Since $B_k$ has $\sigma_k$ columns, Lemma~\ref{Bk_lemma} implies that
$\lambda(n) \leq p^k$ when $n= \sigma_k$ for some nonnegative integer $k$.
We would like a similar upper bound on $\lambda(n)$ for all positive integers $n$.  
Thus, we provide the following proposition.
\begin{proposition}\label{sum_Bk}
If $n = \sum_{j=0}^k b_j \sigma_j$, 
then
$$\lambda(n) \leq \sum_{j=0}^k  b_j p^j.$$
\end{proposition}
\begin{proof}[Proof of Proposition~\ref{sum_Bk}]
We will construct a matrix $M$ with $n$ columns such that $c(M) = \sum_{j=0}^k b_j p^j$.
The matrix $M$ will essentially be a block matrix with $b_j$ copies of $B_j$ for $0 \leq j \leq k$.
However, the number of rows of $B_j$ does not equal the number of rows of $B_\ell$ when $j \neq \ell$.
Thus, for $0 \leq j \leq k$, we define the $(k+1)\times \sigma_j$ matrix $\B{k}{j}$ where the first $j$ rows of $\B{k}{j}$ 
match the first $j$ rows of $B_j$ and the last $k+1-j$ rows of $\B{k}{j}$ all equal the last row of $B_j$.
Thus, $\B{k}{0}$ is a $(k+1)\times 1$ column of 1's, and for $1 \leq j \leq k$, 
$$\B{k}{j} = \left[\begin{array}{c|c|c|c|c}
\text{\Large $B_{j-1}$} & \text{\Large $B_{j-1}$} & \cdots 
& \text{\Large $B_{j-1}$} & \begin{array}{c}0 \\ \vdots \\0 \end{array} \\
\hline
0\ldots 0 & 1\ldots 1 & \cdots & (p-1) \ldots (p-1) & 1 \\
\hline
0\ldots 0 & 1\ldots 1 & \cdots & (p-1) \ldots (p-1) & 1 \\
\hline
 \vdots &    \vdots   &  \cdots  &  \vdots   &  \vdots \\
 \hline
0\ldots 0 & 1\ldots 1 & \cdots & (p-1) \ldots (p-1) & 1 
\end{array}\right]$$
where the last row is repeated $(k+1)-j$ times. 
After comparing $\B{k}{j}$ with $B_j$, it is easy to see that 
$\row^*(\B{k}{j}) = \row^*(B_j)$.

Let $n$ be a positive integer such that $n = \sum_{j=0}^k b_j \sigma_j.$ 
Let $M$ be the $(k+1)\times n$ matrix defined as  
a block matrix with $b_j$ copies of $\B{k}{j}$ for $0 \leq j \leq k$, where the blocks appear in
a single row in
nondecreasing order according to their lower index, i.e.,
$$M = \left[\begin{array}{ccc|ccc|c|ccc}
\undermat{b_0}{\B{k}{0} &\cdots & \B{k}{0}} &
\undermat{b_1}{\B{k}{1} &\cdots & \B{k}{1}} & 
\cdots &
\undermat{b_k}{\B{k}{k} &\cdots & \B{k}{k} } \\
\end{array}
\right].$$
\smallskip

\noindent Let $\vec{v} \in \row^*(M)$.  Then we can (essentially) write 
$$\vec{v} = (\vec{v}_1^{(0)}, \ldots, \vec{v}_{b_0}^{(0)}, \vec{v}_1^{(1)}, \ldots, \vec{v}_{b_1}^{(1)}, \ldots,
\vec{v}_1^{(k)}, \ldots, \vec{v}_{b_k}^{(k)})$$
where $\vec{v}_{i}^{(j)} \in \row^*(B_j)$ for $0 \leq j \leq k$ and $1 \leq i \leq b_j$.
Moreover, for $1\leq i \leq  b_j$, we have $\vec{v}_i^{(j)} = \vec{v}_{b_j}^{(j)}$.
Thus,
$$\|\vec{v}\| = \sum_{j=0}^k b_j \|\vec{v}_{b_j}^{(j)}\|.$$
Because  $\vec{v}_{b_j}^{(j)} \in \row^*(B_j)$, Lemma~\ref{Bk_lemma} implies $\|\vec{v}_{b_j}^{(j)}\| = p^j$, therefore,
$$\|\vec{v}\| = \sum_{j=0}^k b_j p^j.$$
 Thus, $c(M) = \sum_{j=0}^k b_j p^j,$ and $\lambda(n) \leq \sum_{j=0}^k b_j p^j.$
\end{proof}

Thus, we can combine Propositions~\ref{lin_comb} and \ref{sum_Bk} with 
Claim~\ref{n_as_lin_comb_of_sigmas} to obtain the following corollary.
\begin{corollary}\label{lambda_formula}
Let $n \in \integers^+$.  Suppose $n < \sigma_{k+1}$.
Then $$n = \sum_{j=0}^k b_j \sigma_j,$$
where  $0 \leq b_j \leq p$ for $0\leq j \leq k$, and if $b_j = p$, then $b_i = 0$ for $i < j$.
Moreover, 
$$\lambda(n) = \sum_{j=0}^k b_j p^j.$$
\end{corollary}

\begin{corollary}\label{cor_recurr}
The sequence $\lambda(n)$ satisfies the meta-Fibonacci recurrence relation 
$$\lambda(n) = \sum_{i=1}^p  \lambda(n-i+1 -\lambda(n-i)).$$
\end{corollary}
\begin{proof}[Proof of Corollary~\ref{cor_recurr}.]
We refer to Corollary~32 in \cite{rusdeu}, which implies that a sequence which is defined by the
meta-Fibonacci recurrence relation~\eqref{mfp_eqn} is also defined by the recurrence relation
\begin{equation}\label{alt_mfp_eqn}
\lambda(n) = p^k + \lambda(n - \sigma_k), 
\end{equation}
for $\sigma_k \leq n < \sigma_{k+1}$.
Based on Corollary~\ref{lambda_formula}, it is clear that $\lambda(n)$ satisfies recurrence~\eqref{alt_mfp_eqn}.
Therefore, $\lambda(n)$ satisfies the meta-Fibonacci recurrence~\eqref{mfp_eqn}.
\end{proof}


 \end{document}